\documentclass[12pt]{amsart}
\usepackage{amsfonts}
\usepackage{float}
\usepackage{amsmath}
\usepackage{paralist}
\usepackage{amssymb}
\usepackage{amsthm}
\usepackage{amscd}
\usepackage{comment}
\usepackage{listings}
\usepackage{graphicx,mathrsfs}
\usepackage{subfigure}
\usepackage[colorlinks=true]{hyperref}
\hypersetup{urlcolor=blue, citecolor=red}
\usepackage{float}

\graphicspath{./MFGB/}
\allowdisplaybreaks[1]
\numberwithin{equation}{section}
\newtheorem{Thm}{Theorem}[section]

\newtheorem{Prop}{Proposition}[section]

\newtheorem{Asmpt}{Assumption}[section]

\begin{document}
\title{Mean-Field Game for Gene Expression of Beetles}
\author{Yiming Jiang}
\address{School of Mathematical Sciences and LPMC\\ Nankai University\\ Tianjin 300071 China}
\email{ymjiangnk@nankai.edu.cn}
\author{Yuan Lou}
\address{School of Mathematical Sciences\\ Shanghai Jiao Tong University\\ Shanghai 200240 China}
\email{yuanlou@sjtu.edu.cn}
\author{Yawei Wei}
\address{School of Mathematical Sciences and LPMC\\ Nankai University\\ Tianjin 300071 China}
\email{weiyawei@nankai.edu.cn}
\author{Fei Zeng}
\address{School of Mathematical Sciences \\ Nankai University\\ Tianjin 300071 China}
\email{2120230039@mail.nankai.edu.cn}
\author{Zelin Zhang}
\address{School of Mathematical Sciences \\ Nankai University\\ Tianjin 300071 China}
\email{2120240104@mail.nankai.edu.cn}
	
\begin{abstract}
In this paper, we investigate the probability of the expression of genes that control the size of beetles under competitive relationships. We use the mean field game (MFG) theory in multiple populations to characterize the different competitive pressures of large and small beetles in the population, and simulate the probability of gene expression in finite time $[0, T]$. Therefore, we prove the existence and uniqueness of the solution of the equation under some assumptions. 
\end{abstract}
\maketitle
\section{Introduction}

The probability of gene expression in a population is a perennial topic, and when a gene can control a clear characteristic of an organism, it will have a certain impact on the competitive pressure of the organism. Being able to express characteristics that are suitable for the environment first is the key to biological evolution, and judging the tendency of gene expression can help discover the direction of biological evolution.

We referred to (\cite{bt2}, \cite{bt3}, \cite{bt4}) for research on insect habits. In this article, insects, as omnivorous animals, actively pursue prey, obtain resources, and compete with the entire population. Assuming a beetle wants to monopolize the branches and leaves of a plant, and usually needs to compete with other individuals in the population, it may only obtain one-third or even less of it. We record the resources that insects want as their decisions, and the actual resources that insects obtain are smaller than their decisions, depending on the competitive pressure of the population, which is related to the probability of gene expression.

Suppose that there is a gene in the beetles that controls size, which is expressed to make the beetle larger, otherwise smaller. As beetles need to compete with populations in order to obtain resources, this competitive relationship is related to the size of the beetles. Therefore, we divided the beetles into two different populations of different sizes(large and small). Using the MFG method to simulate competition between beetle populations of different sizes, in order to solve the probability of the expression of this gene.

In this article, we consider a multipopulation MFG model that incorporates an unknown function $p(t)$ to couple the equations of two populations. The model is as follows: 
\begin{equation}\label{model}
	\left\{
	\begin{array}{lr}
		-\partial_{t} u_{k} + H_{k}( x , p , \partial_{x} u_{k} ) = \partial_{xx} u_k \\
		\partial_{t} m_k - div_x (m_k D_h H_k)  = \partial_{xx} m_k \\
        \int_{\Omega} D_h H_1 \text{d} m_1 + \int_{\Omega} D_h H_2 \text{d} m_2 = -Q(t) \\
        u_k (T,x) = \overline{u}_k (x), \quad m_k (0,x) = \overline{m}_k (x), 
		
	\end{array}\right.
\end{equation} for $k=1,2$. 

In the previous equation, $x(t)\in \mathbb{R}$ is the state of each beetle at time $t$. The function $u_k(x,t)$ is the value function for a beetle whose resource is $x$ at time $t$. The rate $\alpha$ as the beetles choose to acquire resources. As each beetle should compete with all beetles in the population and how tough competition is in the population is associated with $p$, we use $f_k (p, \alpha)$ to express the resources each beetle chooses to obtain in competition. In addition, beetles should pay some energy $c_0 (\alpha, t)$ to search for resources, since the resources naturally faded as $l(x)$. Let $k =1,2$ to indicate small and larch beetle respectively, we express the rate $c_k (\alpha, p ,x, t)$ each beetle gains resources at time $t$ is
$$ c_k (\alpha, p, x, t) = -c_0 (\alpha, t) + f_k (p, \alpha) - l(x).$$ 
Note that $\alpha \geq 0$ since beetles gain nothing when positively losing resources. The main result is as follows: 
\begin{Thm}
If the function $u_1(x, T), u_2(x, T)\in C^1([0,T], R)$, and $ m_1(x, 0)$, $m_2(x, 0)\in C^1([0,T], \mathbb{P}), Q(t) \in C^1([0,T], R)$ and satisfy Assumption 2.1 and 2.2, then there exists a unique probability function $p(t) \in [0,T] \times [0,1]$ such that the quintuple $(u_1, u_2, m_1, m_2, p)$ is a unique solution of the model
\begin{equation}
	\left\{
	\begin{array}{lr}
		-\partial_{t} u_{1} + H_{1}( x , p , \partial_{x} u_{1} ) = \partial_{xx} u_1 \\
		\partial_{t} m_1 - div_x (m_1 D_h H_1)  = \partial_{xx} m_1 \\
        -\partial_{t} u_{2} + H_{2}( x , p , \partial_{x} u_{2} ) = \partial_{xx} u_2 \\
		\partial_{t} m_2 - div_x (m_2 D_h H_2)  = \partial_{xx} m_2 \\
        \int_{\Omega} D_h H_1 \text{d} m_1 + \int_{\Omega} D_h H_2 \text{d} m_2 = -Q(t) \\
        u_k (T,x) = \overline{u}_k (x), \quad m_k (0,x) = \overline{m}_k (x), 
		
	\end{array}\right.
\end{equation} for $k=1,2$. 
\end{Thm}

In this model, the Hamiltonian we are studying has three variables: $x$, $p$, and $\partial_x u_k$. This is different from the classical MFG model. In Section 2, we will explain this Hamiltonian by deducing the model. 

Mean field game theory which is devoted to solve optimal control problems with large number of rational players has been developed by Lasry and Lions in series of papers (\cite{LL1}, \cite{LL2}, \cite{LL3}, \cite{LL4}).
Energy formation models are sort of price formation models using MFG theory in \cite{LL3}. This type of research has been advanced by many different researchers (\cite{pr1}, \cite{pr3}). In this article, we assume price is a given function of gene expression probability, and decisions making simultaneously affects prices and demand.

In mathematical physics, research for the behavior of a large number of identical particles has been developed in \cite{psy1}, \cite{psy2}. 
Related ideas has been developed independently at same time in series of papers by Huang-Caines-Malhame(\cite{HMC1}, \cite{HMC2}, \cite{HMC3}, \cite{HMC4}). 
In case of application, study the numerical approximation of the solution of MFG models: see Achdou and Capuzzo Dolcetta \cite{A1}, Achdou, Camilli and Capuzzo Dolcetta \cite{A2}. 
The mean field games theory seem also paticularly adapted to modelize problems in economics: see Gu´eant \cite{eco1}, \cite{eco2}.  

In this article, our main achievement is to complete the proof of the existence and uniqueness of the solution for this model. At the same time, we provide a specific example of this equation and subsequently perform numerical simulations on it, presenting some specific function graphs to demonstrate the practical application value of our model.

Remind that $x(t)$ is a numerical value used to express the true ability of beetles to obtain resources. By incorporating natural noise, we have
$$ \text{d} x(t) = c_k (t) \text{d} t + \sqrt{2} \text{d} W(t),$$ 
the $W(t)$ is standard Brownian motion in possibility space $(\Omega_p, \mathcal{F}, P)$. 

Next, we consider the cost $c_k$ that beetles need to spend on making decisions at $(x, t)$, as well as the cost function $J_k$ from time $t$ to $T$ $(t < T)$. $k = 1, 2$ represents two different populations of beetles: large beetles and small beetles. The cost function $J_k$ is 
$$ J_k (\alpha, p, x, t) = \int_t^T -c_k(\alpha, p, x, s) \text{d}s - \int_t^T \sqrt{2} \text{d}W(s) + \overline{u_k}(x), $$   
where $J_k(\alpha, p, x, t)$ describes the cost that a single beetle needs to spend from time $t$ to $T$, and $\overline{u_k}(x)$ is the boundary condition.
The value function 
$$ u_k(x,t) = \min\limits_\alpha \mathbb{E} J_k (\alpha, p, x, t) ,$$ 
For any $t<T$ and $\delta>0$ small enough, we have 
$$ u_k(x(t),t) = \min\limits_\alpha [ \int_t^{t+\delta t} -c_k(s) \text{d}s + u_k(x(t+\delta t), t+\delta t)].$$
Calculate, we have
\begin{equation}\label{u_1}
    \partial_t u_k + \min\limits_\alpha [-c_k + \partial_x u_k c_k] = -\partial_{xx} u_k.
\end{equation}
For the Hamiltonian, 
\begin{align}
H_k \nonumber & = -\sup\limits_\alpha (c_k - \partial_x u_k c_k) \\
    \nonumber & = -\sup\limits_\alpha [ (1 - \partial_x u_k) (-c_0(\alpha , t) + f_k(p , \alpha) - l(x))] \\
    \nonumber & = -(1-\partial_x u_k) (F_k(p)- l(x)).
\end{align}
So we note Hamiltonian as $H_k(x, p, \partial_x u_k)$, and \eqref{u_1} tends to
$$-\partial_t u_k + H_k(x, p, \partial_x u_k) = \partial_{xx} u_k.$$

As a comparison, for the conventional Hamiltonian in MFG models, 
$$H(x,p) = \sup\limits_{\alpha \in A} [ r(x, \alpha) + f(x, \alpha) p], $$ 
where $A$ is the set of decision. That is to say, in this model, the condition we studying is $r(x, \alpha)=f(x, \alpha)=c_k(x, \alpha, p)$, and the $p$ in this model is the $\partial_x u_k$ in ours. Since we want to obtain the function $p(t)$, we have to represent $p$ in the model. From here, we can see the difference in our equations, which is also a research difficulty: the variable $p (t)$ we want to find is in the running cost $r(x, \alpha)$. Therefore, it is a difficult proposition to separate and solve $p$, which we will address in Section 3.1. 

For the Fokker-Planck equation, the optimal is $c_k^* = - D_h H_k(x,p,h)$, 
$$ \partial_{t} m_k - div_x (m_k D_h H_k)  = \partial_{xx} m_k.$$ 
The resource changing of two populations at time t is 
$$\int_{\Omega} D_h H_1 \text{d} m_1 + \int_{\Omega} D_h H_2 \text{d} m_2 = -Q(t).$$
The above constitutes equation \eqref{model}. 

For the Hamiltonian, we need to control the increase of $F(p)$. And since $l(x)$ represents the natural decomposition rate of resources, we assumpt $l(x)$ is some increase linear function of $x$, with coefficient less than 1 and not too low. So we assumptions are: 
\begin{Asmpt}
The Assumptions we need are as follows:

    \text{1).}  The Hamiltonian $H_k$ is 
    $$ H_k(x, p, h) = -\sup\limits_\alpha \big[ (1 - h) (-c_0(\alpha , t) + f_k(p , \alpha) - l(x)) \big].$$
    where $f_k \in C^2(\mathbb{R}^+ \times [0,T])$, $c_0 \in C^2(\mathbb{R}^+ \times [0,T])$ and $l(x) \in C^2(\mathbb{R})$. $p \rightarrow H_k$ is Lipschitz by some constant $C > 0$, and $l(x) = a_0x + a_1$ where $\frac{1}{2} < a_0 < 1$.
    
    \text{2).} The terminal condition $\overline{u}_k (x) $ and the rate function $ l(x) $ are semiconcave.

    \text{3).} The terminal condition $\overline{u}_k (x) $ and the rate function $ l(x) $ is Lipschitz under constant $1-\delta$ for some $0<\delta<1$.

    \text{4).} Suppose that $D_{pp}^2 H_k \leq 0$, $D_p H_k(p,h) \leq D_{pp}^2 H_k(p,h)$ and $D_p H_k \leq 2H_k$. 
\end{Asmpt}

    Noted that $$(1-h)D_h H_k = -H_k,$$ and $D_x H_k$ is not related to $p$, $D_h H_k$ is not related to $h$. What's more, $D_{hx}^2 H_k$ is a positive constant. 

    Since $p \rightarrow H_k$ is Lipschitz, it implies that $p \rightarrow F_k(p)$ and $p \rightarrow D_h H_k(p)$ is Lipschitz. 

To obtain the fixed point, we need condition of rate function $l(x)$ and the terminal condition $\overline{u}_k (x)$. 

Then, to obtain the conclusion of uniqueness, we need the following. Since the natural decline rate of resources is slower than the growth of resources.

Finally, Since large beetles have more affection in the environment, cost function changes intensely at the beginning of $p$ from $0$ to $1$. Thus we assumpt $H$ is concave about $p$, and we need to control the increase of $H_k$.

\section{Main Results}

This paper explains the model in chapter 1, as well as some assumptions needed to prove the existence and uniqueness, and explains the rationality of the assumptions. The first half of chapter 2 proves the existence of the solution, and the second half proves the uniqueness of the solution. The solution of the MFG equation, which is a simulated value function, has been discussed in relevant articles: see \cite{pr1}, \cite{pr2}, \cite{pr3}. However, because of the Hamiltonian we focusing in this paper is different from the classical MFG equation, which also leads to computational complexity, especially in terms of uniqueness. Chapter 3 explains a specific function as a Hamiltonian and proves that it satisfies the assumptions which tends to the existence and uniqueness of solution in this case. Chapter 4 is a summary of the paper. 

In this chapter, our main goal is to solve the problem of the existence and uniqueness of the model \ref{model}. In Section 2.1, 
We will use Schauder's fixed point theorem to solve the problem of the existence of model solutions. For this, we need to: 

1. Separate the probability function $p (t)$ from the equation and replace it with $\theta$; 

2. Prove the continuity of $\theta \rightarrow p$ mapping. 

In Section 2.2, we will provide a conclusion on the uniqueness of the solution by calculating the monotonicity of the operator.

\subsection{Existence of a Solution}
We know that the MFG equation is derived from the coupling of the Hamilton-Jacobi-Bellman equation and the Fokker-Planck equation. In this model, The Hamilton-Jacobi-Bellman equation
\begin{equation}\label{HJB}
	\left\{
	\begin{array}{lr}
		-\partial_{t} u_{k} + H_{k}( x , p , \partial_{x} u_{k} ) = \partial_{xx} u_k \\
        u_k (T,x) = \overline{u}_k (x),
		
	\end{array}\right.
\end{equation} for $k=1,2$,
and the Fokker-Planck equation
\begin{equation}\label{FP}
	\left\{
	\begin{array}{lr}
		\partial_{t} m_k - div_x (m_k D_h H_k)  = \partial_{xx} m_k \\
        \quad m_k (0,x) = \overline{m}_k (x), 
		
	\end{array}\right.
\end{equation} for $k=1,2$.
If we fix a $p(t): [0,T]\rightarrow [0,1]$ , by general stochastic optimal control theory, there is a unique viscosity solution $u_k(x, t)$ of the HJB equation \eqref{HJB}. And the Fokker-Planck equation has a unique solution $m(t,x) = \mathcal{L}(x)$. First, we need to prove that $p\to u_i$, $i=1,2$ are continue.
\begin{Prop}
    If Assumption 2.1(2) holds, then $x\to u_k(x,t)$ is semiconcave and the semiconcave constant is independent of $p$ for $k=1,2$. 
\end{Prop}
\begin{proof}
Expanding $u_k(x,t)$, we have 
\begin{align}
\nonumber    u_k(x,t) & = \min\limits_\alpha \mathbb{E} J_k (\alpha, p, x, t) \\
   & = \min\limits_\alpha [\int_t^T c_0 (\alpha, s) - f_k (p, \alpha) + l(x) \text{d}s + \overline{u_k}(x)]  \\
\nonumber    & = \int_t^T c_0 (\alpha^*, s) - f_k (p, \alpha^*) + l(x) \text{d}s + \overline{u_k}(x),
\end{align}
where $\alpha^*$ is the optimal control in $(x,t)$.
For any $h > 0$, we have 
$$u_k(x\pm h,t) \leq \int_t^T c_0 (\alpha^*, s) - f_k (p, \alpha^*) + l(x\pm h) \text{d}s + \overline{u_k}(x\pm h).$$
As Assumption 2.2, both $\overline{u}_k (x)$ and $ l(x) $ are semiconcave. Then there is a constant $C$ such that
$$ \overline{u_k}(x + h) + \overline{u_k}(x - h) - 2\overline{u_k}(x) \leq Ch^2, $$
$$ l(x + h) + l(x - h) - 2l(x) \leq Ch^2, $$
$$ u_k(x + h) + u_k(x - h) - 2u_k(x) \leq Ch^2. $$
So $u_k$ is semiconcave. 
\end{proof}

\begin{Prop}
    If any given $p(t)$, $H_k$ satisfies the Assumption 2.1(1), then equation \eqref{HJB} has a unique viscosity solution $u$. Moreover, if Assumption 2.1-2.2 holds and $p_n$ uniformly converges to $p$, then $u_n$ uniformly converges to $u$, and $\partial_x u_k^n$ converges to $\partial_x u_k$ almost everywhere.
\end{Prop}

\begin{proof}

    Remind that $u_n$ uniformly converges to $u$ by the property of viscosity solution. As $u_k$ is semiconcave, fix $x\in \mathbb{R}$, 
$$\left|\partial_x u_k^n(x) - \partial_x u_k (x)\right| = \left|\lim_{h \to 0}\dfrac{u_k^n(x+h)-u_k^n(x)-u_k(x)+u_k(x-h)}{h}\right|.$$ 
Since $u_n$ uniformly converges to $u$ , for any $\epsilon>0$, $|u_k^n(x) - u_k(x)|<\epsilon$ is true for $n$ large enough. Then take $\epsilon = \delta h^2$, we have 
\begin{align}
\nonumber    |\partial_x u_k^n(x) - \partial_x u_k (x)| &=\left|\lim_{h \to 0}\dfrac{u_k(x+h)-2u_k(x)+u_k(x-h)}{h}\right| \\
\nonumber    &= \lim_{h \to 0} \big(|Ch| + |2\delta h|\big), 
\end{align}
where $C$ is the semiconcavity constant of $u_k$. 

\end{proof}

Second, we need to prove that $p\to u_i$, $i=1,2$ are continue. 
\begin{Prop}
    For a given p, with Assumption 2.1(1) holds, then equation \eqref{FP} has a unique solution $m_k$ and satisfies 
    $$ d_1( m_k(t), m_k(t+a) ) \leq C \sqrt{a},$$
    where $d_1$ is the 1-Wasserstein distance, and C is independent of p. 
\end{Prop}
\begin{proof}
We have 
\begin{align}
\nonumber    d_1( m_k(t), m_k(t+a) ) &= \sup\{ \int_{\mathbb{R}} \phi(x)(m_k(t) - m_k(t+a))\text{d}x \} \\
\nonumber    &\leq \sup \{\mathbb{E} [\phi(x(t))-\phi(x(t+a))]\} \\
\nonumber    &\leq \mathbb{E} [|x(t) - x(t+a)|] \\
\nonumber    &\leq \mathbb{E} [\int_t^{t+a} c_k(s,x,\alpha)\text{d}s + \sqrt{2}|W(t) - W(t+a)|] \\
\nonumber    &\leq ||c_k||_{\infty}a + \sqrt{2a},
\end{align}
where $\phi(x)$ is 1-Lipschitz continues.
    
\end{proof}

\begin{Prop}
    If it is assumed that Assumption 2.1(1)-(2) holds and $p^n$ uniformly converges to $p$, corresponding to $u_k^n$ and $u_k$ being the solutions of equation \eqref{HJB} (for $k=1,2$), then $m_k^n$ converges to $m_k$ (for $k=1,2$).
\end{Prop}
\begin{proof}
    For any given $p^n$, according to Proposition 3.3, the corresponding unique solution $\{m_k^n\}$ is equicontinuous. And $\{m_k^n\}$ is also uniformly bounded, therefore, according to the Arzila-Ascoli Theorem, $\{m_k^n\}$ converges to a certain point $m_k$. We then need to prove that $m_k$ is the solution of the equation \eqref{FP} for given $p$. For any test function $\psi(x)$, we have
    $$ \int_0^T\int_{\mathbb{R}} \partial_x \psi D_hH_k(x,p^n,\partial_x u_k^n)m_k^n\text{d}x\text{d}t \to  \int_0^T\int_{\mathbb{R}} \partial_x \psi D_hH_k(x,p,\partial_x u_k)m_k\text{d}x\text{d}t,$$
    since $u_k^n \to u_k$ almost everywhere. 
\end{proof}

Third, For the equation
$$\int_{\mathbb{R}} D_h H_1 m_1 \text{d} x + \int_{\mathbb{R}} D_h H_2 m_2 \text{d}x = -Q(t),$$
if Assumption 2.1 is assumed to hold, there exists a unique initial value $p(0)$ that makes the equation
$$\sum_{i=1,2}\int_{\mathbb{R}} D_h H_i(x, p(0), \partial_x u_i(0))m_i(0,x) \text{d}x = -Q(0)$$
established.

Further, 
\begin{align}
\nonumber \int_{\mathbb{R}} D_h H_1 m_1 \text{d} x + \int_{\mathbb{R}} D_h H_2 \text{d} m_2 & = -Q(t).
\end{align}
Differentiating both sides with $t$, we have
$$ \sum_{i=1,2}\int_{\mathbb{R}} (D_{hx}^2 H_i \dot{x} + D_{hp}^2 H_i \dot{p}) m_i + D_h H_i \partial_t m_i \text{d}x = -\dot{Q}(t).$$
As $\partial_{t} m_i = div_x (m_i D_h H_i) + \partial_{xx} m_i$, we have
$$ \sum_{i=1,2}\int_{\mathbb{R}} (D_{hx}^2 H_i c_i^* + D_{hp}^2 H_i \dot{p}) m_i + D_h H_i [div_x (m_i D_h H_i) + \partial_{xx} m_i] \text{d}x = -\dot{Q}(t),$$
$$D_h H_i[div_x (m_i D_h H_i) + \partial_{xx} m_i] = D_h H_i[\partial_x m_i D_h H_i + m_i D_{hx}^2 H_i + \partial_{xx} m_i].$$
Noted that
$$ (D_h H_i D_h H_i m_i)_x = 2D_{hx}^2 H_i D_h H_i m_i + D_h H_i D_h H_i \partial_x m_i,$$
and
$$ (D_h H_i \partial_x m_i)_x = D_{hx}^2 H_i \partial_x m_i + D_h H_i \partial_{xx} m_i.$$
So
\begin{align}
\nonumber D_h H_i[div_x (m_i D_h H_i) + \partial_{xx} m_i] & = D_h H_i(-D_{hx}^2 H_i m_i + \partial_{xx} m_i) \\
\nonumber & = -D_{hx}^2 H_i (D_h H_i m_i + \partial_x m_i).
\end{align}
Noted that $c_k^* = -D_h H_i,$ so
$$\sum_{i=1,2}\int_{\mathbb{R}} -D_{hx}^2 H_i (2D_h H_i m_i + \partial_x m_i) + D_{hp}^2 H_i \dot{p} m_i \text{d}x = -\dot{Q}(t).$$
Then
$$ \dot{p} = \dfrac{-\dot{Q}(t) + \sum_{i=1,2}\int_{\mathbb{R}}D_{hx}^2 H_i (2D_h H_i m_i + \partial_x m_i)}{\sum_{i=1,2}\int_{\mathbb{R}}D_{hp}^2 H_i m_i \text{d}x}.$$

\begin{equation}\label{resource}
	\left\{
	\begin{array}{lr}
		\dot{\theta} = \dfrac{-\dot{Q}(t) + \sum_{i=1,2}\int_{\mathbb{R}}D_{hx}^2 H_i (2D_h H_i m_i + \partial_x m_i)\text{d}x}{\sum_{i=1,2}\int_{\mathbb{R}}D_{hp}^2 H_i m_i \text{d}x} \\
        \theta_0 = p(0).
		
	\end{array}\right.
\end{equation}

Finally, we prove that $p\to \theta$ is continue. 
\begin{Prop}
    If it is assumed that Assumption 2.1(1)-(2) holds and $p^n$ uniformly converges to $p$, corresponding to $u_k^n$ and $u$ being the solutions of equation \eqref{HJB}, and $m_k^n$ converges to $m$ in \eqref{FP} then $\theta^n$ uniformly converges to $\theta$. 
\end{Prop}
\begin{proof}
    Noted that $\theta(t+a) - \theta(t) = \int_t^{t+a} \dot{\theta}\text{d}t$, and the right side of equation \eqref{resource} is bounded, the $\{\theta^n\}$ is equicontinuous and uniform bounded. So by Ascoli-Arzila Theorem, $\{\theta^n\}$ converges to some point $\theta$. Using the same discussion in Proposition 3.4, we can conclude that $\theta$ is the solution of equation \eqref{resource} for given $p$. And then $\theta^n \to \theta$ uniformly. 
\end{proof}

That shows, the function $p \rightarrow \theta$ continues. As both $p$ and $\theta : [0,T] \to [0,1]$ are bounded and closed, by Schauder fixed-point theorem, there is a fixed point $p$, and $(u_1, u_2, m_1, m_2, p)$ solves equation \eqref{model}. 

\subsection{Uniqueness}
For the uniqueness of equation \eqref{model}, we can see the equation \eqref{model} as a operator $T$ acting on a 5-tuple element $(m_1, m_2, u_1, u_2, p)$. 

We begin with a 5-tuple element $y = (m_1, m_2, u_1, u_2, p)$, in some certain region $D$. 
And 
\begin{align}
    \nonumber  Ty &= 
\begin{pmatrix}
    \partial_{t} u_{1} - H_{1}( x , p , \partial_{x} u_{1} ) + \partial_{xx} u_1 \\
    \partial_{t} u_{2} - H_{2}( x , p , \partial_{x} u_{2} ) + \partial_{xx} u_2 \\
    \partial_{t} m_1 - div_x (m_1 D_h H_1)  - \partial_{xx} m_1 \\
    \partial_{t} m_2 - div_x (m_2 D_h H_2)  - \partial_{xx} m_2 \\
    \int_{\Omega} D_h H_1 \text{d} m_1 + \int_{\Omega} D_h H_2 \text{d} m_2 + Q(t) \\
\end{pmatrix} \\
\nonumber  &=
\begin{pmatrix}
    \partial_{t} u_{1} + \partial_{xx} u_1 \\
    \partial_{t} u_{2} + \partial_{xx} u_2 \\
    \partial_{t} m_1 - \partial_{xx} m_1 \\
    \partial_{t} m_2 - \partial_{xx} m_2 \\
    0 \\
\end{pmatrix}
+
\begin{pmatrix}
    - H_{1}( x , p , \partial_{x} u_{1} ) \\
    - H_{2}( x , p , \partial_{x} u_{2} ) \\
    - div_x (m_1 D_h H_1) \\
    - div_x (m_2 D_h H_2) \\
    \int_{\Omega} D_h H_1 \text{d} m_1 + \int_{\Omega} D_h H_2 \text{d} m_2 + Q(t) \\
\end{pmatrix} \\
\nonumber  &= T_1 y + T_2 y, 
\end{align}
where $T_1$ is linear and $T_2$ is nonlinear.

\begin{Prop}
    If Assumption 2.1(3) holds, then $x\to u_k(x,t)$ is Lipschitz with constant $1-\delta$ for some $0<\delta<1$. 
\end{Prop}
\begin{proof}
Expanding $u_k(x,t)$, we have 
\begin{align}
\nonumber    u_k(x,t) & = \min\limits_\alpha \mathbb{E} J_k (\alpha, p, x, t) \\
   & = \min\limits_\alpha [\int_t^T c_0 (\alpha. t) - f_k (p, \alpha) + l(x) \text{d}s + \overline{u_k}(x)]  \\
\nonumber    & = \int_t^T c_0 (\alpha^*. t) - f_k (p, \alpha^*) + l(x) \text{d}s + \overline{u_k}(x),
\end{align}
where $\alpha^*$ is the optimal control in $(x,t)$.
For any $h > 0$, we have 
$$u_k(x+h,t) \leq \int_t^T c_0 (\alpha^*. t) - f_k (p, \alpha^*) + l(x+h) \text{d}s + \overline{u_k}(x+h).$$
As Assumption 2.3, both $\overline{u}_k (x)$ and $ l(x) $ is Lipschitz. Then 
$$ \overline{u_k}(x + h) -\overline{u_k}(x) \leq (1-\delta)h, $$
$$ l(x + h)-l(x) \leq (1-\delta)h, $$
$$ u_k(x + h)-u_k(x) \leq (1-\delta)h. $$
So $u_k$ is Lipschitz with constant $1-\delta$. 
\end{proof}

Proposition 3.6 proves that if Assumption 2.1(3) holds, then we have $\partial_x u_i(x,t)\leq 1-\delta$. 

For any $y^1, y^2 \in D$, $y^1 = (m_1, m_2, u_1, u_2, p)$, $y^2 = (m_1', m_2', u_1', u_2', p')$, calculate that $(Ty^1-Ty^2, y^1-y^2)$ with inner product 
$$((a_1,b_1,c_1,d_1,f_1), (a_2,b_2,c_2,d_2,f_2)) = \int_{[0,T]\times \mathbb{R}} a_1a_2 + b_1b_2 +c_1c_2 +d_1d_2 +f_1f_2 \text{d}x\text{d}t,$$ 
we have
\begin{align}
\nonumber    (T_1y^1-T_1y^2, y^1-y^2) =&\sum_{i=1,2}\int_{[0,T]\times \mathbb{R}}(m_i-m_i')[\partial_{t}(u_i-u_i') + \partial_{xx} (u_i-u_i')] \\
\nonumber    &+(u_i-u_i')[\partial_{t} (m_i-m_i') - \partial_{xx} (m_i-m_i')]\text{d}x\text{d}t \\
             =&\sum_{i=1,2}\int_{[0,T]\times \mathbb{R}}(m_i-m_i')\partial_{t}(u_i-u_i') + (u_i-u_i')\partial_{t} (m_i-m_i') \\
\nonumber    &+(m_i-m_i')\partial_{xx} (u_i-u_i') - (u_i-u_i')\partial_{xx} (m_i-m_i')\text{d}x\text{d}t \\
\nonumber    =&0.
\end{align}

Since $u_i-u_i'$ and $m_i-m_i'$ is $0$ respectively at time $0$ and $T$. Then, for the nonlinear operator $T_2$, we have

\begin{align}
\nonumber    &(T_2y^1-T_2y^2, y^1-y^2) \\
\nonumber  =&-\sum_{i=1,2}\int_{[0,T]\times \mathbb{R}}\big(H_{i}( x , p , \partial_{x} u_{i} )-H_{i}( x , p' , \partial_{x} u_{i}' )\big)(m_i-m_i') \\
\nonumber   &-div_x \big[m_i D_h H_i(x,p,\partial_{x}u_{i}) - m_i' D_h H_i(x,p',\partial_{x}u_{i}')\big](u_i-u_i')\\
\nonumber   &+ \big[D_h H_i(x,p,\partial_{x}u_{i}) m_i- D_h H_i(x,p',\partial_{x}u_{i}') m_i\big](p-p')\text{d}x\text{d}t \\
\nonumber  =&-\sum_{i=1,2}\int_{[0,T]\times \mathbb{R}}\big(H_{i}( x , p , \partial_{x} u_{i} )-H_{i}( x , p' , \partial_{x} u_{i}' )\big)(m_i-m_i') \\
\nonumber  &-\big[m_i D_h H_i(x,p,\partial_{x}u_{i}) - m_i' D_h H_i(x,p',\partial_{x}u_{i}')\big](\partial_xu_i-\partial_xu_i')\\
\nonumber   &+ \big[D_h H_i(x,p,\partial_{x}u_{i}) m_i- D_h H_i(x,p',\partial_{x}u_{i}') m_i'\big](p-p')\text{d}x\text{d}t \\
\nonumber  =&\sum_{i=1,2}\Big\{\int_{[0,T]\times \mathbb{R}}m_i\big[H_i(x,p',\partial_{x}u_{i}')-H_i(x,p,\partial_{x}u_{i}) \\
\nonumber  &- (\partial_xu_i'-\partial_xu_i + p'-p)D_h H_i(x,p,\partial_{x}u_{i})\big] \\
\nonumber  &+m_i'\big[H_i(x,p,\partial_{x}u_{i})-H_i(x,p',\partial_{x}u_{i}') \\
\nonumber  &- (\partial_xu_i-\partial_xu_i' + p-p')D_h H_i(x,p',\partial_{x}u_{i}')\big]\text{d}x\text{d}t\Big\}.
\end{align}

Note that 
$$H_i(x,p,\partial_{x}u_{i}')-H_i(x,p,\partial_{x}u_{i}) =(\partial_xu_i'-\partial_xu_i)D_hH_i(x,p,\partial_{x}u_{i}),$$
and
\begin{align}
\nonumber H_i(x,p',\partial_{x}u_{i}')-H_i(x,p,\partial_{x}u_{i}') = &(p'-p)D_pH_i(x,p,\partial_{x}u_{i}') \\
\nonumber &+ (p'-p)^2 D_{pp}^2 H_i(x,p_\epsilon,\partial_{x}u_{i}'),
\end{align}
where $p_{\epsilon}(t)$ take some value between $p(t)$ and $p'(t)$ for every $t \in [0,T]$. 
So 
\begin{align}
\nonumber  &\sum_{i=1,2}\Big\{\int_{[0,T]\times \mathbb{R}}m_i\big[H_i(x,p',\partial_{x}u_{i}')-H_i(x,p,\partial_{x}u_{i}) \\
\nonumber  &- (\partial_xu_i'-\partial_xu_i + p'-p)D_h H_i(x,p,\partial_{x}u_{i})\big] \\
\nonumber  &+m_i'\big[H_i(x,p,\partial_{x}u_{i})-H_i(x,p',\partial_{x}u_{i}') \\
\nonumber  &- (\partial_xu_i-\partial_xu_i' + p-p')D_h H_i(x,p',\partial_{x}u_{i}')\big]\text{d}x\text{d}t\Big\} \\
\nonumber  =&\sum_{i=1,2}\Big\{\int_{[0,T]\times \mathbb{R}}m_i\big[(p'-p)D_pH_i(x,p,\partial_{x}u_{i}') + (p'-p)^2 D_{pp}^2 H_i(x,p_\epsilon,\partial_{x}u_{i}') \\
\nonumber  &-(p'-p)D_h H_i(x,p,\partial_{x}u_{i})\big] \\
\nonumber  &+m_i'\big[(p-p')D_pH_i(x,p',\partial_{x}u_{i}) + (p-p')^2 D_{pp}^2 H_i(x,p_\epsilon,\partial_{x}u_{i}) \\
\nonumber  &-(p-p')D_h H_i(x,p',\partial_{x}u_{i}')\big]\text{d}x\text{d}t\Big\}. \\
\nonumber =&\sum_{i=1,2}\Big\{\int_{[0,T]}\int_0^\infty (p'-p)\big[m_i D_p H_i(x,p,\partial_x u_i') - m_i' D_p H_i(x,p',\partial_x u_i)\big] \\
\nonumber &+(p'-p)\big[m_i' D_h H_i(x,p',\partial_x u_i') - m_i D_h H_i(x,p,\partial_x u_i)\big] \\
\nonumber &+(p'-p)^2\big[m_i D_{pp}^2 H_i(x,p_\epsilon,\partial_{x}u_{i}') + m_i' D_{pp}^2 H_i(x,p_\epsilon,\partial_{x}u_{i})\big] \text{d}x\text{d}t\Big\}.
\end{align}

By the above formula, we find that if $p' = p$, then $(Ty^1-Ty^2, y^1-y^2) = 0$.Which means, if $p$ is unique, then the solution of model \eqref{model} is unique. Thus, we only need proposition below to prove the uniqueness of solution. 
\begin{Prop} The ordinary differential equation \eqref{resource}:
\begin{equation}
\nonumber	\left\{
	\begin{array}{lr}
		\dot{p} = \dfrac{-\dot{Q}(t) + \sum_{i=1,2}\int_{\mathbb{R}}D_{hx}^2 H_i (2D_h H_i m_i + \partial_x m_i)}{\sum_{i=1,2}\int_{\mathbb{R}}D_{hp}^2 H_i m_i \text{d}x} \\
        p_0 = p(0).
		
	\end{array}\right.
\end{equation}
has unique solution in $t\in[0,T]$. 
\end{Prop}
\begin{proof}
    Let 
    $$R(p,t) = \dfrac{-\dot{Q}(t) + \sum_{i=1,2}\int_{\mathbb{R}}D_{hx}^2 H_i (2D_h H_i m_i + \partial_x m_i)\text{d}x}{\sum_{i=1,2}\int_{\mathbb{R}}D_{hp}^2 H_i m_i \text{d}x}, $$ 
    and let $R: J \times S \rightarrow E$ with $J = [0, T], S = \{p\in E:||p-p_0||\leq 1\}$
    For two solution of \eqref{resource} $p$ and $p'$, calculate $||R(p',t) - R(p,t)||$. Noted that $D_{hp}^2 H_i$ is not related to $x$, and $D_{hpp}^3 H_i\geq 0$, let $p^*(t) = \min[p(t), p'(t)]$, for any $t\in [0,T]$. 
    
    Using Assumption 2.1(1) and (4) of $H$, we have $p \rightarrow R(p, t)$ is increase. And reminding that $D_{hx}^2 H_i$ is a positive constant $\frac{1}{2} < a_0 < 1$, Thus 
    \begin{align}
    \nonumber    &\|R(p',t) - R(p,t)\| \\
    \nonumber    \leq &\|\dfrac{\sum_{i=1,2}\int_{\mathbb{R}}D_{hx}^2 H_i (2D_h H_i(x,p') m_i' - 2D_h H_i(x,p) m_i + \partial_x m_i' - \partial_x m_i)\text{d}x}{\sum_{i=1,2} D_{hp}^2 H_i(p^*)}\| \\
    \nonumber    = &\|\dfrac{\sum_{i=1,2}\{[2D_{hx}^2 H_i(F_i(p') - F_i(p))] + \int_{\mathbb{R}}D_{hx}^2 H_i (\partial_x m_i' - \partial_x m_i)\text{d}x\}}{\sum_{i=1,2} D_{hp}^2 H_i(p^*)}\| \\
    \nonumber    \leq &\Big|\dfrac{\sum_{i=1,2}2D_{hx}^2 H_i}{\sum_{i=1,2} D_{hp}^2 H_i(p^*)}C\Big| \|p'(t) - p(t)\|.
    \end{align}
    By [\cite{NFA},Theorem 5.2.1] with constant $K = \Big|\dfrac{\sum_{i=1,2}[2D_{hx}^2 H_i]}{\sum_{i=1,2} D_{hp}^2 H_i(p^*)}C\Big|$, the equation \eqref{resource} has unique solution $p$ in $t\in [\frac{T}{2}-\delta, \frac{T}{2}+\delta]$, where $0<\delta<\min\{\frac{T}{2}, \frac{1}{M}, \frac{1}{K}\}$ and $M = \sup \limits_{(p,t)} ||R(p, t)||$. Since $p \rightarrow R(p, t)$ is local Lipschitz, there is a unique solution of \eqref{resource} in $t\in [0, T]$ by [\cite{NFA}, Theorem 5.2.2]. 
\end{proof}

The proposition above tends to $p=p'$. For the uniqueness of the solution of the HJB equation and the Fokker-Planck equation, we have $u_1=u_1'$, $u_2=u_2'$, $m_1=m_1'$, $m_2=m_2'$. Therefore, the solution of equation \eqref{model} is unique. 

\section{Example}
In this section, we focus on a specific function $f_k(p,\alpha)=b_kp\alpha$, $c_0(\alpha, t)=a\alpha^2$ and $l(x)=cx$, $k=1,2$, where $a$ and $b$ are constant satisfying Assumptions 2.1-2.4. And study the related equation with discussion similar to the previous sections. For the Hamiltonian, 
\begin{align}
\nonumber    H_k &= -\sup\limits_\alpha [ (1 - \partial_x u_k) (-a\alpha^2 + b_kp\alpha - cx)] \\
\nonumber    &=-(1 - \partial_x u_k) (\frac{b_k^2p^2}{4a} - cx).
\end{align}
Since
$$D_h H_k = \frac{b_k^2p^2}{4a} - cx,$$ 
we suppose that $0\leq p_1 \leq p_2 \leq 1$, then 
$$ |D_h H_k(p_2)-D_h H_k(p_1)| = |(p_2-p_1)b_k^2\frac{p_1+p_2}{4a}|\leq \big|\frac{b_k^2}{2a}\big||p_2-p_1|.$$ 
Then $D_h H_k$ is Lipschitz with constant $C=|\dfrac{b_k^2}{2a}|$. $D_{pp}^2 H_k \leq 0$, $D_p H_k(0,p,h) \leq D_{pp}^2 H_k(0,p,h)$ and $D_p H_k \leq 2H_k$ are easy to confirm. And $l(x)$ is increase linear function of $x$. Fix $\overline{u}_k (x)$ such that $\overline{u}_k (x)$ is semiconcave, and Lipschitz with $1-\delta$, then all the Assumption 2.1-2.4 are satisfied. Then there is a unique solution of model \eqref{model}.

In the two figures below, we take initial $p(0)$ from $0$ to $1$ in different interval $0.1$ and $0.01$. We can find that in the end of the model \ref{model}, the gene expression $p$ at time $T$ are all around $0.25$. 
In this numerical simulation, we also assume many initial conditions, such as the initial condition $m_k(x,0)$ is a Gaussian distribution (boundary condition), etc. In general, to some extent, the numerical simulation shows that this model has certain stability and reference value.

We set the initial density distribution as normal, with a = 1.0, b1 = 1.0, b2 = 1.2, c = 0.5, using the above example, we can get the following two numerical simulation pictures: 

\begin{figure}[htbp]
    \centering
    \includegraphics[width=0.6\linewidth]{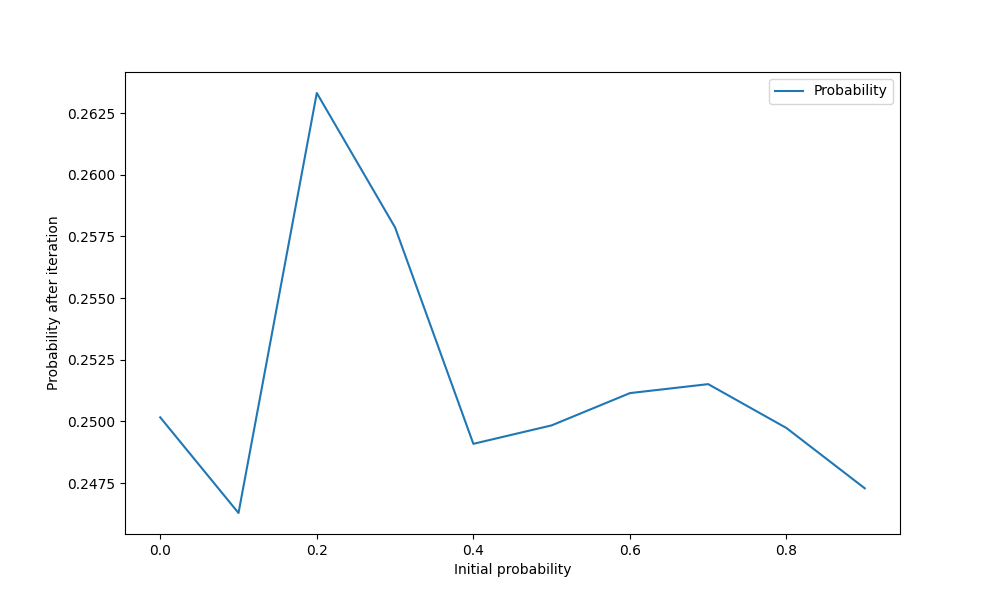}
    \caption{$p(0)-p(T)$ with interval 0.1}
\end{figure}
\begin{figure}
    \centering
    \includegraphics[width=0.6\linewidth]{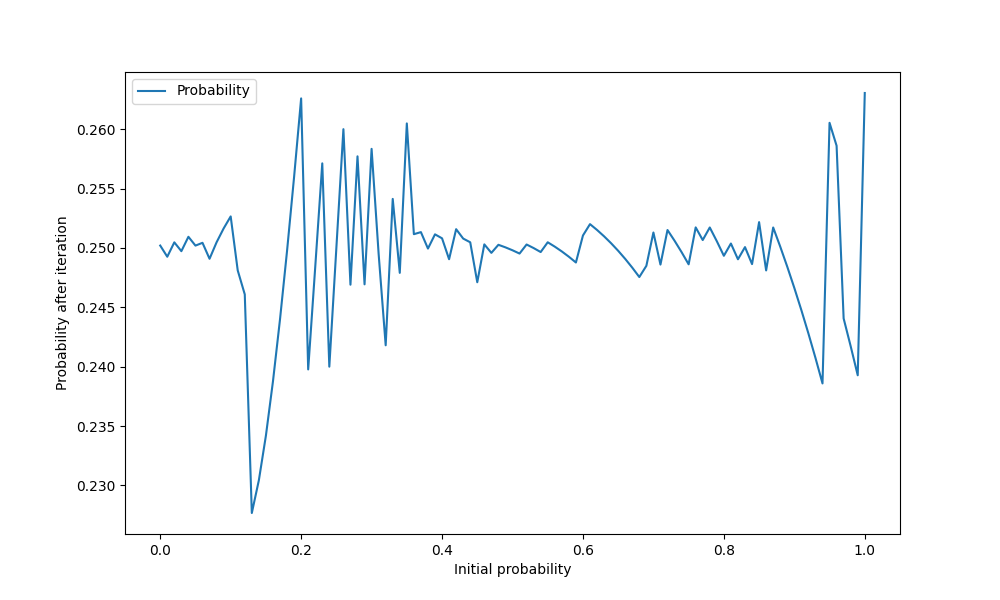}
    \caption{$p(0)-p(T)$ with interval 0.01}
    \label{fig:enter-label}
\end{figure}

\section{Conclusions}
We used multiple population MFG theories to describe the competitive relationship between beetles of different sizes in model as follows: 
\begin{equation}\label{model}
	\left\{
	\begin{array}{lr}
		-\partial_{t} u_{k} + H_{k}( x , p , \partial_{x} u_{k} ) = \partial_{xx} u_k \\
		\partial_{t} m_k - div_x (m_k D_h H_k)  = \partial_{xx} m_k \\
        \int_{\Omega} D_h H_1 \text{d} m_1 + \int_{\Omega} D_h H_2 \text{d} m_2 = -Q(t) \\
        u_k (T,x) = \overline{u}_k (x), \quad m_k (0,x) = \overline{m}_k (x), 
		
	\end{array}\right.
\end{equation} for $k=1,2$. 
And proved the existence and uniqueness of the expression probability of genes controlling beetle size in the population competition. This illustrates that under established competition rules and with continuous energy growth, the ratio of large beetles to small beetles within the population can converge to a consistent level. In other words, such a population will not be invaded by other populations with different size ratios of beetles, because over time, their ratio always stabilizes at a unique fixed point. Similar conclusions can be directly extended to other populations, or even two entirely distinct populations, provided that two conditions are met: they primarily compete with each other within a certain region, and they operate under established competition rules with an external energy growth function.

Broadly speaking, if a gene controls a certain characteristic of a certain organism and affects the competitive pressure function, then under certain assumptions, we can still obtain the unique fixed point of the ratio of big and small beetles, which means the probability of gene expression can tend towards a stable value. 

\textbf{Conflict of Interest}
The authors declare that they have no conflict of interest.


\addcontentsline{toc}{section}{References} 
\end{document}